\documentclass[12pt]{amsart}
\usepackage{fullpage}
\newtheorem{theorem}{Theorem}
 
\theoremstyle{definition}
\newtheorem{definition}[theorem]{Definition}
\newtheorem{remark}[theorem]{Remark}
\newcommand{\id}{\operatorname{id}}
\newcommand{\rank}{\operatorname{rank}}

\newcommand{\R}{{\mathbb R}}
\newcommand{\C}{{\mathbb C}}

\newcommand{\F}{{\mathbb F}}
\newcommand{\Q}{{\mathbb Q}}

\title{Matrix Completion and Tensor Rank}
\author{Harm Derksen}
\thanks{The author was partially supported by NSF grant DMS 0901298.}
\begin{document}

\begin{abstract}
 In this paper, we show that the low rank matrix completion problem can be reduced
to the problem of finding the rank of a certain tensor.
\end{abstract}
\maketitle
\section{introduction}
Suppose we are given a rectangular array which has only been partially filled out with entries in $\R$
(or some other field).
The low rank matrix completion problem asks to fill out the remaining entries such that
the resulting matrix has minimal rank.
For example,  the array
$$
\begin{pmatrix}
2 & \cdot & \cdot \\
2 & 3 & \cdot\\
\cdot & 6 & 2\end{pmatrix}
$$
can be completed to a rank $1$ matrix
$$
\begin{pmatrix}
2 & 3 & 1 \\
2 & 3 & 1\\
4 & 6 & 2\end{pmatrix}
$$
but clearly cannot be completed to a rank $0$ matrix. For a survey on matrix completion problems, see~\cite{Johnson, Laurent}.
The low rank matrix completion problem has various
applications, such as statistics, computer vision, signal processing and control. The low rank matrix completion problem is known to be NP-hard (see~\cite[Theorem 3.1]{Peeters}). 
For the field $\F=\R$ or $\F=\C$
one sometimes can solve the matrix completion problem using convex relaxation: 
 under certain assumptions minimizing the nuclear norm of the matrix (the sum of the singular values)
 one also obtains the minimal rank completion. See for example \cite{RFP, CR, CT,KMO}. The rank of a tensor was defined in \cite{Hitchcock}. Using this notion, 
 matrix completion problems  can be generalized to higher order tensors (see~\cite{GRY}).

Suppose that $\F$ is a field, and $V^{(i)}=\F^{m_i}$ for $i=1,2,\dots,d$. The tensor product space
$$
V=V^{(1)}\otimes V^{(2)}\otimes \cdots \otimes V^{(d)}
$$
can be viewed as a $d$-dimensional array of size $m_1\times m_2\times \cdots \times m_d$. A {\em pure} tensor is an element of $V$ of the form
$$
v^{(1)}\otimes v^{(2)}\otimes \cdots \otimes v^{(d)},
$$
where $v^{(i)}\in V^{(i)}$ for all $i$. The rank of a tensor $T\in V$ is the smallest nonnegative integer $r$
such that $T$ can be written as a sum of $r$ pure tensors:
$$
T=\sum_{i=1}^r v^{(1)}_i\otimes v^{(2)}_i\otimes \cdots \otimes v^{(d)}_i
$$
where $v^{(i)}_j\in V^{(i)}$ for all $i$ and $j$. We denote the tensor rank of $T$ by $\rank(T)$.
If $A=(a_{i,j})$ is an $n\times m$ matrix,
then we can identify $A$ with the tensor
$$
\sum_{i=1}^n\sum_{j=1}^m a_{i,j} (e_i\otimes e_j),
$$
 in $\F^n\otimes \F^m$,
where $e_i$ denotes the $i$-th basis vector. The tensor rank in this case is exactly the same as the rank of the matrix. For $d\geq 3$ it is often difficult to find the rank of a given tensor.
Finding the rank of a tensor is NP-complete if the field  $\F$ is  finite  (\cite{Hastad,Hastad2}) and
NP-hard if the field is $\Q$ (\cite{Hastad,Hastad2}), or if it contains $\Q$ (\cite{HL}). Approximation
of a tensor by another tensors of low rank is known as the PARAFAC (\cite{Harshman}) or CANDECOMP (\cite{CC}) model. There are various algorithms for finding low-rank approximations. See~\cite{KB,TB} for a discussion.
The problems of finding the tensor rank of a tensor, and to approximate tensors with tensors of low rank
 has many applications, such as the complexity of matrix multiplication, fluorescence spectroscopy,  statistics, psychometrics, geophysics and magnetic resonance imaging. 
In general, the rank of higher tensors more ill-behaved than the rank of a matrix (see~\cite{DSL}).

In the next section we will show that the low rank matrix completion problem can be reduced to finding
the rank of a certain tensor. Lek-Heng Lim pointed out to the author, that because low rank matrix completion is NP-hard, this gives another proof that determining the rank of a tensor  is NP-hard.

\section{The main result }
Suppose that $\F$ is a field, and $A=(a_{i,j})$ is an $n\times m$ array for which some of the entries are given, and
the entries in positions $(i_1,j_1),\dots,(i_s,j_s)$ are undetermined. Define $e_{i,j}$ as the matrix
that has a $1$ in position $(i,j)$ and $0$ everywhere else.
We can reformulate the low rank matrix completion problem as follows. Define $a_{i_k,j_k}=0$ for $k=1,2,\dots,s$ and find a matrix $B$
of the form
$$
B=A+\sum_{k=1}^s \lambda_k e_{i_k,j_k},\quad \lambda_1,\dots,\lambda_s\in \F.
$$
which has minimal rank.

We can view matrices as second order tensors. So we identify $e_{i,j}$ with $e_{i}\otimes e_j\in \F^n\otimes \F^m$ and we have
$$
A=\sum_{i=1}^n\sum_{j=1}^m a_{i,j}(e_i\otimes e_j).
$$
Define $u_k=e_{i_k}$ and $v_k=e_{j_k}$.
The matrix completion problem asks to find $\lambda_1,\dots,\lambda_s\in \F$ such that the tensor
$$
B=A+\sum_{k=1}^s \lambda_k (u_k\otimes v_k)=\sum_{i=1}^n\sum_{j=1}^m a_{i,j}(e_i\otimes e_j)+\sum_{k=1}^s \lambda_k (u_k\otimes v_k)
$$
has minimal rank.

Let $U\subseteq \F^{n}\otimes \F^m$ be the vector space spanned by $u_1\otimes v_1,\dots,u_s\otimes v_s$.
\begin{definition}
We define
$$
\rank(A,U)=\min\{\rank(B)\mid B\in A+U\}.
$$
\end{definition}
So the matrix completion asks for the value of $\rank(A,U)$ and to find a matrix $B\in A+U$ for which $\rank(A,U)=\rank(B)$.

We define a third order tensor $\widehat{A}\in \F^n\otimes \F^m\otimes \F^{s+1}$
by
\begin{equation}\label{eq:Ahat}
\widehat{A}=A\otimes e_{s+1}+\sum_{k=1}^s u_k\otimes v_k\otimes e_k=
\sum_{i=1}^n\sum_{j=1}^m a_{i,j} (e_i\otimes e_j\otimes e_{s+1})+\sum_{k=1}^s u_k\otimes v_k\otimes e_k.
\end{equation}
Our main result is:
\begin{theorem}\label{theo:main}
We have
$$
\rank(A,U)+s=\rank(\widehat{A}).
$$
\end{theorem}
The theorem is effective: given an explicit decomposition of $\widehat{A}$ as a sum of $l=\rank(\widehat{A})$ pure tensors, 
the proof shows that one easily can construct a matrix $B$ in $A+U$ such that $\rank(B)=l-s=\rank(A,U)$.
\begin{proof}[Proof of Theorem~\ref{theo:main}]
Let $r=\rank(A,U)$ and $l=\rank(\widehat{A})$.
By definition, there exists a matrix $B\in A+U$  with rank $r$. We can write
$$
B=A+\sum_{k=1}^s \lambda_k (u_k\otimes v_k).
$$
Since $B$ has rank $r$, we also can write
$$
B=\sum_{i=1}^r x_i\otimes y_i
$$
for some $x_1,\dots,x_r\in \F^n$ and $y_1,\dots,y_r\in \F^m$.
We have 
\begin{multline*}
\widehat{A}=A\otimes e_{s+1}+\sum_{k=1}^s u_k\otimes v_k\otimes e_k=
(B-\sum_{k=1}^s \lambda_k u_k\otimes v_k)\otimes e_{s+1}+\sum_{k=1}^s u_k\otimes v_k\otimes e_k=\\=
B\otimes e_{s+1}+\sum_{k=1}^s u_k\otimes v_k\otimes (e_k -\lambda_k e_{s+1})=
\sum_{i=1}^r x_i\otimes y_i+\sum_{k=1}^s u_k\otimes v_k\otimes (e_k-\lambda_k e_{s+1}).
\end{multline*}
We have written $\widehat{A}$ as a sum of $r+s$ pure tensors. So $l=\rank(\widehat{A})\leq r+s$.

Conversely, we can write
$$
\widehat{A}=\sum_{j=1}^l a_j\otimes b_j\otimes c_j
$$
 For every $i$, choose a linear function $h_i:\F^n\otimes \F^m\to \F$ such that $h_i(u_i\otimes v_i)=1$
and $h_i(u_j\otimes v_j)=0$ if $j\neq i$.
Applying $h_i\otimes \id:\F^n\otimes \F^m\otimes \F^{s+1}\to \F^{s+1}$ to the tensor $\widehat{A}$ gives
$$
\sum_{j=1}^l h_i(a_j\otimes b_j) c_j=(h_i\otimes \id)(\widehat{A})=e_i+h_i(A) e_{s+1}.
$$
This shows that $e_i$ lies in the span of $c_1,\dots,c_l,e_{s+1}$ for all $i$. So $c_1,\dots,c_l,e_{s+1}$
span the vector space $\F^{s+1}$.
After rearranging $c_1,\dots,c_l$, we may assume that $c_1,c_2,\dots,c_{s},e_{s+1}$ is a basis of $\F^{s+1}$.
There exists a linear function $f:\F^{s+1}\to \F$ such that $f(c_1)=\cdots=f(c_s)=0$ and $f(e_{s+1})=1$.
We apply $\id\otimes \id\otimes f:\F^n\otimes \F^m\otimes \F^{s+1}\to \F^n\otimes \F^m$ to $\widehat{A}$:
$$
A+\sum_{i=1}^s f(e_j) u_i\otimes v_i=(\id\otimes\id\otimes f)(\widehat{A})=\sum_{j=1}^l f(c_j) a_j\otimes b_j=\sum_{j=s+1}^l f(c_j)a_j\otimes b_j.
$$
The tensor above is a matrix in $A+U$ which has rank  at most $l-s$, and by definition of $r=\rank(A,U)$
it has rank at least $r$.
It follows that $r\leq l-s$. We have already proven that $r+s\geq l$, so we conclude that $r+s=l$.
\end{proof}

\begin{remark}
In the proof we have not used that $u_k$ and $v_k$ are basis vectors. So Theorem~\ref{theo:main}
is true whenever $U$ is the span of linearly independent pure tensors
$$u_1\otimes v_1,\dots,u_s\otimes v_s$$
and $\widehat{A}$ is given by (\ref{eq:Ahat}).
\end{remark}
\begin{remark}
Define 
$$t_i\in (\F^2)^{\otimes s}=\underbrace{\F^2\otimes \cdots\otimes \F^2}_s
$$
by 
$$
t_i=e_1\otimes \cdots \otimes e_1\otimes e_2\otimes e_1\otimes \cdots \otimes e_1
$$
where $e_2$ appears in the $i$-th position if $i\leq s$, and
$$
t_{s+1}=e_1\otimes e_1\otimes \cdots \otimes e_1.
$$
We define a tensor $\widetilde{A}\in \F^n\otimes \F^m\otimes (\F^2)^{\otimes s}$ by
$$
\widetilde{A}=u_1\otimes v_1\otimes t_1+\cdots+u_s\otimes v_s\otimes t_s+A\otimes t_{s+1}.
$$
A proof, similar to that of Theorem~\ref{theo:main} shows that $\rank(\widetilde{A})=s+\rank(A,U)$.
The ambient vector space of $\widetilde{A}$ is much larger than that of $\widehat{A}$. But one can
imagine that using the tensor $\widetilde{A}$ can be advantageous for proving lower bounds
for $\rank(A,U)$ because it is, in a way, more rigid.
\end{remark}
\begin{remark}
Theorem~\ref{theo:main} also easily generalizes to the higher order analog of matrix completion: {\em tensor completion}. Suppose that $V^{(1)},\dots,V^{(d)}$ are finite dimensional $\F$-vector spaces and
$
V=V^{(1)}\otimes \cdots \otimes V^{(d)}$.
Assume that $U\subseteq V$ is a subspace spanned by $s$ linearly independent pure tensors
$$
u^{(1)}_i\otimes u^{(2)}_i\otimes \cdots \otimes u^{(d)}_i,\quad i=1,2,\dots,s
$$
and 
$$A\in V^{(1)}\otimes \cdots \otimes V^{(d)}$$
is a tensor. As before, we define
$$
\rank(A,U)=\min\{\rank(B)\mid B\in A+U\}.
$$
Define a $(d+1)$-th order tensor  by
$$
\widehat{A}=A\otimes e_{s+1}+\sum_{k=1}^s u^{(1)}_k\otimes \cdots \otimes u^{(d)}_k\otimes e_k\in V^{(1)}\otimes \cdots \otimes V^{(d)}\otimes \F^{s+1}.
$$
Then we have
$$\rank(A,U)+s=\rank(\widehat{A}).
$$
\end{remark}
\subsection*{Acknowledgment}
The author thanks Lek-Heng Lim for some useful suggestions, and Reghu Meka for providing me with an important reference.

\ \\[20pt]
\noindent{\sl Harm Derksen\\
Department of Mathematics\\
University of Michigan\\
530 Church Street\\
Ann Arbor, MI 48109-1043, USA\\
{\tt hderksen@umich.edu}}

 \end{document}